\documentclass[a4paper]{article}

\usepackage{amsmath}
\usepackage{amsthm}
\usepackage{amsfonts}
\newlength\PullBackLength
\newcommand\PullBack[1][2]{%
  \setlength{\global\PullBackLength}{#1em}%
  \kern\PullBackLength%
  &
  \kern-\PullBackLength}
\makeatletter
\@addtoreset{equation}{section}
\makeatother

\newtheorem{theorem}{Theorem}
\newtheorem{lemma}[theorem]{Lemma}
\theoremstyle{definition}
\newtheorem{problem}[theorem]{Problem}

\newtheorem{definition}[theorem]{Definition}
\newcommand{\normalvec}[1]{\overset{\,\rightarrow}{\raisebox{0pt}[6pt][0pt]{$#1$}}}
\DeclareMathOperator{\diam}{diam}

\title{The Complexity of Two Graph Orientation Problems}
\author{Nicole Eggemann\thanks{Supported by the EC Marie Curie programme NET-ACE (MEST-CT-2004-6724).}\mbox{ } and Steven D. Noble\thanks{Partially supported by the Heilbronn Institute for Mathematical Research, Bristol, UK.}\\Department of Mathematical Sciences, Brunel University\\Kingston Lane\\Uxbridge\\UB8 3PH\\United Kingdom\\ nicole.eggemann@brunel.ac.uk and steven.noble@brunel.ac.uk}

\begin{document}
\maketitle

\begin{abstract}
We consider two orientation problems in a graph, namely the minimization of the sum of all the shortest path lengths and the minimization of the diameter.
We show that it is NP-complete to decide whether a graph has an orientation such that the sum of all the shortest paths lengths is at most an integer specified in the input. The proof is a short reduction from a result of Chv\'atal and Thomassen showing that it is NP-complete to decide whether a graph can be oriented so that its diameter is at most $2$. In contrast, for each positive integer $k$, we describe a linear-time algorithm that decides for a planar graph $G$ whether there is an orientation for which the diameter is at most $k$. We also extend this result from planar graphs to any minor-closed family $\mathcal F$ not containing all apex graphs.
\end{abstract}

\section{Introduction}
We consider two problems concerned with orienting the edges of an undirected graph in order to minimize two global measures of distance in the resulting directed graph.
Our work is motivated by an application involving the design of urban light rail networks of the sort described in~\cite{lowson:10}. In such an application, a number of stations are to be linked with unidirectional track in order to minimize some function of the travel times between stations and subject to constraints on cost, engineering and planning. In practice these constraints mean that the choice of which stations to link may be forced upon us and the only control we have is over the choice of direction of each piece of track. Since the stations that are linked tend to be those that are close to each other, we make the simplifying assumption that the travel time along each single piece of track or link is the same. Consequently the network can be viewed as an (unweighted) graph in which the vertices represent stations and the edges represent track that is to be built. Furthermore planning constraints tend to rule out the possibility of tracks crossing, so the graph is usually planar. The aim is to orient the resulting graph to minimize travel time. We assume that each journey in the oriented network progresses along the shortest directed path from the vertex representing the starting station to the vertex representing the destination.

All our graphs are simple, that is they have no loops or parallel edges. When the underlying graph is obvious, we use $n$ and $m$ to denote its numbers of vertices and edges, respectively. We use $\normalvec{G}$ to denote a directed graph obtained by orienting the edges of $G$. We let $d(x,y)$ denote the distance from vertex $x$ to vertex $y$ in a directed graph. The two measures of the quality of an orientation are its \emph{diameter} $\diam(\normalvec{G})$, given by $\diam(\normalvec{G})=\max_{x\not= y} d(x,y)$
and the \emph{Wiener Index} $Z(\normalvec{G})$, given by
$Z(\normalvec{G})=\sum_{x\not=y} d(x,y)$.
The name Wiener Index is perhaps not widely used but is more common in applications in chemistry.

The networks arising in the application tend to be planar and have small degree. Our original aim was to determine the complexity of minimizing $\diam(\normalvec{G})$ and $Z(\normalvec{G})$ for planar graphs of bounded degree. We have two partial results in this direction.
In the next section we show that the problem where we are given a graph $G$ and an integer $k$ and must determine whether $G$ can be oriented so that its Wiener Index is at most $k$, is NP-complete. We believe that it should be possible to sharpen this result so that the input graph is restricted to having degree at most three but our idea for a proof became extremely complicated, so we have not pursued this. We do not know whether the input may be restricted to planar graphs. Our hardness result depends on a result of Chv\'atal and Thomassen~\cite{chvatal:78} who showed that determining whether a graph may be oriented so that its diameter is at most two is NP-complete.

In contrast, a result of Bollob\'as and Scott~\cite{bollobas+scott:07-diam2} shows that an oriented graph with diameter two and $n$ vertices must have at least $(1+o(1))n\log_2 n$ edges. Since a planar graph with $n$ vertices has at most $3n-6$ edges, this implies that there is a constant upper bound on the number of vertices in a planar oriented graph with diameter two. So there is a constant time algorithm to determine whether a planar graph can be oriented so that its diameter is at most two. However there are arbitrarily large planar graphs that can be oriented so that their diameter is three, for example, a set of triangles sharing a common edge.

The bulk of this paper is devoted to showing that for any fixed integer $l$, there is a polynomial time algorithm that will take a planar graph and determine whether it may be oriented so that its diameter is at most $l$. In fact we establish rather more than this. An \emph{apex} graph is a graph $G$ having a vertex $v$ such that $G-v$ is planar.
Using a result of~\cite{eppstein:diameter}, we can extend our main result to apply not just to planar graphs but to any minor closed family of graphs not containing all apex graphs.

In Section~\ref{sectree} we discuss necessary concepts from tree-width. In Section~\ref{secalg} we describe an algorithm that attempts to find a suitable orientation when the input graph has bounded tree-width.
Section~\ref{secres} contains our main result.
The main reason that our algorithm works is that the diameter of a graph does not increase when an edge is contracted and is at least $\Omega(l)$ for an $l\times l$-grid. Such a parameter is essentially what is called contraction-bidimensional in~\cite{demaine+fomin:06_bidimensional_local_treewidth,demaine+hajiaghayi:08_diameter,demaine+hajiaghayi:08_linearity_of_grid_minors}, where a general framework is described for when the corresponding decision problems for these parameters are tractable. Perhaps the most notable example is finding a $k$-dominating set in a planar graph~\cite{alber:02_domination,fomin+thilikos:06_domination_planar}.

\section{Complexity of the Wiener Index}\label{sec:relprob}
Imagine we are given a graph and an integer $k$ and we would like to know whether the graph can be oriented in such a way that the Wiener Index is less than $k$.
In this section we investigate the complexity of this problem.

\begin{sloppy}
Chv\'atal and Thomassen~\cite{chvatal:78} showed that the following problem is NP-complete.
\begin{problem}\label{prob:chva}
\mbox{ }\\
Instance: A graph $G$.\\
Question: Is it possible to orient $G$ to ensure that $\diam(\normalvec{G})\le 2$?
\end{problem}
\end{sloppy}

From this result we can easily conclude that the following problem concerning the Wiener Index is NP-complete.
\begin{problem}\label{prob:wien}
\mbox{ }\\
Instance: A graph $G$, integer $k$.\\
Question: Is it possible to orient $G$ to ensure that the Wiener Index of $\normalvec{G}$ is at most $k$?
\end{problem}
\begin{theorem}
Problem~\ref{prob:wien} is NP-complete.
\end{theorem}
\begin{proof}
The problem is clearly in NP. Suppose that $G$ has $m$ edges. Let $k=2(n^2-n)-m$. If $\diam(\normalvec{G})\le 2$ then all pairs of vertices are either joined by an edge or by a path of length two. So $Z(\normalvec{G})=2(n^2-n)-m=k$. Conversely if $\diam(\normalvec{G})>2$, there are $n^2-n-m$ pairs of vertices joined by paths of length at least two including at least one path of length at least three, so $Z(\normalvec{G})>2(n^2-n)-m=k$.
Consequently there is an orientation of $G$ with $\diam(\normalvec{G})\le 2$ if and only if there is an orientation of $G$ with $Z(\normalvec{G}) \le k$.
\end{proof}

We have been unable to determine the complexity of the following problem.
\begin{problem}\label{pro:unsolved}
\mbox{ }\\
Instance: Planar graph $G$ and integer $k$.\\
Question: Can we orient the edges of $G$ so that $Z(\normalvec{G}) \le k$?
\end{problem}

The rest of the paper is dedicated to the investigation of the complexity of the following problem for any fixed integer $l$ and any minor-closed family $\mathcal F$ of graphs not containing all apex graphs.

\begin{problem}\label{prob:orient}
\mbox{ }\\
Instance: Graph $G$ belonging to $\mathcal F$. \\
Question: Can we orient the edges of $G$ so that $\diam(\normalvec{G}) \le l$?
\end{problem}

There is a considerable amount of work devoted to orienting graphs to minimize the diameter, see for instance the survey~\cite{koh:02_survey} or the book~\cite{bang-Jensen+Gutin:book}. Much of the focus has been on very specific classes of graphs. One algorithmic result is that for $l \ge 4$, it is NP-complete to determine whether a chordal graph has an orientation of diameter at most $l$~\cite{fomin+matamala+rapaport:04}.

\section{Tree-decompositions}\label{sectree}
The notion of a tree-decomposition was developed by Robertson and Seymour in~\cite{robertson+seymour:gm3}. Good introductions to the theory of tree-decompositions can be found, for example, in~\cite{bod:touristguide,lovasz:05,thomas:_treedec}. The definition of a tree-decomposition is as follows.
\begin{definition}
A tree-decomposition $\mathcal{T}$ of a graph $G$ is a pair $(T,\mathcal{W})$ where $T$ is a tree and $\mathcal{W}=(W_t:t\in V(T))$ is a family of subsets of $V(G)$ such that:
\begin{itemize}
\item $\bigcup_{t \in V(T)}W_t=V(G)$ and every edge in $G$ has both endpoints in $W_t$ for some $t$;
\item if $t,t',t'' \in V(T)$ and $t'$ lies on the path from $t$ to $t''$ in $T$ then $W_t \cap W_{t''} \subseteq W_{t'}$.
\end{itemize}
The width of $(T,\mathcal{W})$ is defined to be
\begin{eqnarray*}
\max\{|W_t|-1:t \in V(T)\}.
\end{eqnarray*}
The tree-width of $G$ is the minimum width among all possible tree-\newline decompositions of $G$.
\end{definition}

One reason for the importance of tree-width is that many NP-hard problems can be solved in polynomial or even linear time when restricted to graphs of bounded tree-width~\cite{bod:touristguide,thomas:_treedec}.

Bodlaender~\cite{bod:linalg} gave a linear-time algorithm for finding tree decompositions of small width.
\begin{theorem}\label{linalg}
For any positive integer $k$, there is an algorithm running in time $O\big(2^{\theta(k^3)}n\big)$ that inputs a graph $G$ and determines whether the tree-width of $G$ is at most $k$, and if so finds a tree-decomposition of $G$ with width at most $k$.
\end{theorem}

\section{Minimizing diameter: the bounded tree-width case}\label{secalg}

In this section we show that there is an algorithm which for fixed integers $k,l$, takes as input a graph $G$ and a tree-decomposition with width $k$ and determines whether there is an orientation of $G$ with diameter at most $l$. The algorithm runs in time $O(cn)$, where $c$ is a constant depending on $k$ and $l$.

It is possible to construct such an algorithm explicitly. For a full description see~\cite{eggemann:thesis} or for a brief outline see~\cite{eggemann-noble:ENDM}. There, it is shown that $c$ may be taken to be
\[ (l+1)^{2(k+1)^2}2^{[4(l+1)^{k+1}+2(l+1)^{2k+2}]}k^2.\]

Given that a full description of the algorithm is extremely lengthy and that the constant is so large, we do not describe the explicit algorithm here but instead prove the existence of a linear time algorithm by using the theory of monadic second-order logic of graphs (MSOL) introduced by Courcelle in~\cite{courcelle:MSOL1}.

We briefly give some background on MSOL here but for more information see~\cite{courcelle:HandbookTCS} or~\cite{courcelle:MSOL1}.
A graph $G$ is represented by a triple $<V,E,R>$ where $V$ and $E$ are just the vertex and edge sets respectively of $G$ and $R \subset V \times E$ is a relation with $(v,e)\in R$ if $v$ is an endpoint of $e$.
An MSOL formula on $<V,E,R>$
 may contain \emph{member variables}, denoting members of either $V$ or $E$, and \emph{set variables}, denoting subsets of either $V$ or $E$. The atomic formulae of an MSOL formula on $<V,E,R>$ are $v\in U$, $e\in A$, $v=w$, $e=f$ and $(v,e)\in R$ where $v,w$ are variables denoting vertices, $e,f$ are variables denoting edges, $U$ denotes a set of vertices and $A$ denotes a set of edges. Standard logical connectives are permitted and both existential and universal quantification is allowed over both types of variable.

Courcelle~\cite{courcelle:MSOL3} showed that for any graph property $\mathcal P$, that may be expressed by an MSOL formula, for each $k$, the problem of determining which graphs satisfy $\mathcal P$ is solvable in time $O(|E||V|)$ when the input is restricted to graphs having tree-width at most $k$. Arnborg, Lagergren and Seese~\cite{arnborg:MSOL} gave another proof of this result, reducing the time bound to $O(|V|)$.

\begin{theorem}\label{th:bt}
For any $k$ and $l$, there exists an algorithm
that takes as input a graph $G$ with tree-width at most $k$ and a tree decomposition of $G$ with width at most $k$, and determines whether $G$ can be oriented so that its diameter is at most $l$. The algorithm runs in time $O(c_{k,l}n)$ where $c_{k,l}$ depends only on $k$ and $l$.
\end{theorem}
\begin{proof}
The proof just consists of showing that the property of having an orientation of diameter at most $l$ is expressible in MSOL. Clearly it is necessary for $G$ to be connected in order to have an orientation with finite diameter. Since connectivity is easily expressed in MSOL we shall assume that $G$ is indeed connected.

Let $\Omega(G)$ denote the set of all orientations of the edges of $G$. From a naive perspective, a formula to check that a graph $G$ may be oriented so that its diameter is at most $l$, should have the form $\exists \omega \in \Omega(G): F(\omega,G)$ where $F$ is a formula that is satisfied if and only if $\omega$ is an orientation with the required properties.

However, it is not clear whether it is possible to quantify over all orientations in such a straighforward way because an orientation is essentially a relation on the ordered pairs of endpoints of each edge
and quantification over functions and relations on the vertices and edges is not allowed in MSOL. However, as long as it is possible to specify the existence of one arbitrary or base orientation $\omega_0$ then it becomes possible to effectively quantify over them all, by quantifying over all subsets $A$ of $E$, and reversing the orientation of edges in $A$ with respect to $\omega_0$.

Courcelle~\cite{courcelle:MSOL8} showed that it is possible to choose a base orientation
by taking a depth-first search tree. More precisely there is an MSOL formula $\theta(X,r,x,y)$ such that for each edge $xy$ we have either $\theta(X,r,x,y)$ or $\theta(X,r,y,x)$ but not both.
(Here $X$ must be a set of edges forming a spanning tree and $r$ must be a vertex representing the root of the tree.)

Now let
\begin{align*}
\PullBack{g(u,v,A,X,r)= (u \in V) \wedge (v \in V) \wedge \big[u=v \vee \exists e: e \in E \wedge uRe \wedge vRe\wedge} \\
&\big((\theta(X,r,u,v) \wedge e \not \in A) \vee (\theta(X,r,v,u) \wedge e \in A)\big)\big].
\end{align*}
So $g$ determines if either $u=v$ or there is an edge from $u$ to $v$ in the orientation formed by reversing the edges of $A$ with respect to the orientation specified by $\theta$.

We can now express the property of having an orientation of diameter at most $l$ by
\begin{align*}
\PullBack{\exists X \exists r \exists A: \forall x \forall v_0 \forall v_l [(x\in A \rightarrow x\in E)}  \\
&\wedge (( v_0\in V) \wedge (v_l\in V) \rightarrow \exists v_1, \ldots ,v_{l-1}:\\
&(g(v_0,v_1,A,X,r) \wedge g(v_1,v_2,A,X,r) \wedge  \cdots \wedge g(v_{l-1},v_l,A,X,r)))]
\end{align*}
\end{proof}
In the next section we shall see how to extend this result to certain classes of graphs containing members with arbitrarily large tree-width.

\section{Minimising the diameter when an apex graph is excluded}\label{secres}
It is now straightforward to establish our main result using the following restatement of a theorem of Eppstein~\cite{eppstein:diameter}.
\begin{theorem}\label{th:eps}
If $\mathcal{F}$ is a minor-closed family of graphs that does not include all apex graphs, then there is a function $f$ such that any graph $G\in \mathcal F$ with diameter at most $d$ has tree-width at most $O(f(d))$.
\end{theorem}
\begin{theorem}
For any minor-closed family $\mathcal{F}$ of graphs that does not include all apex graphs and for every $l$, Problem~\ref{prob:orient} is solvable in time $O(cn)$ where $c$ depends only on $l$ and $\mathcal{F}$.
\end{theorem}
\begin{proof}
Fix $\mathcal F$ and $l$. Then, by Theorem~\ref{th:eps}, there exists $k$ (depending only on $\mathcal F$ and $l$) such that any graph in $\mathcal F$ with diameter at most $l$ has tree-width at most $k$.
So given an input graph $G$, using Bodlaender's algorithm we can determine whether $G$ has tree-width at most $k$ and if so find a tree decomposition with width at most $k$ in time $O(c'n)$ where $c'$ depends only on $k$. If $G$ has tree-width at most $k$ then Theorem~\ref{th:bt} implies that there is an algorithm to determine whether $G$ may oriented to have diameter at most $l$ running in time $O(c''n)$ where $c''$ depends only on $k$ and $l$. On the other hand if the tree-width of $G$ exceeds $k$ then its diameter exceeds $l$ and so it cannot be oriented to have diameter at most $l$.
\end{proof}
A consequence of this result is that for fixed $\mathcal F$ not containing all apex graphs, Problem~\ref{prob:orient} is fixed parameter tractable with respect to $l$.

Originally our main aim was to establish this result for the special case where $\mathcal F$ is the class of planar graphs. Because of this and the fact that we can obtain an expression for the constant in the running time bound, we give a sketch proof of this special case that does not use Theorem~\ref{th:eps}.

\begin{lemma}\label{le:diamh}
Any planar graph having a $(2l+1)\times (2l+1)$-grid-minor has diameter at least $l$.
\end{lemma}

\begin{proof}
Suppose that $G$ is a planar graph having a $(2l+1)\times (2l+1)$-grid-minor. We may assume that $G$ is connected because otherwise $\diam(G)=\infty$. So a $(2l+1)\times (2l+1)$-grid may be obtained from $G$ by a sequence of contractions followed by a series of deletions of edges.
If an edge of a graph is contracted then its diameter cannot increase. Let $K$ be the graph obtained from $G$ after all the contractions of edges, so $\diam(G) \ge \diam(K)$.

It follows from Whitney's Theorem~\cite{whitney:33_2-isomorphism} that a $(2l+1)\times(2l+1)$-grid has a unique embedding on a sphere. $K$ is a simple planar graph of which the $(2l+1)\times (2l+1)$-grid is a spanning subgraph. The only edges of $K$ that are not present in the grid must have both endpoints in the same face of the grid. Consequently $\diam(G)\geq\diam(K) \geq l$ and so the diameter of $G$ is at least $l$.
\end{proof}

We need the following result from~\cite{robertson+seymour+thomas:quickly-excluding-planar-graph}.

\begin{theorem}\label{minor}
Any planar graph with no $g \times g$-grid-minor has tree-width at most $6g-5$.
\end{theorem}

We can now establish the result for planar graphs.
\begin{theorem}
For every $l$, Problem~\ref{prob:orient} is solvable in time
\[O\big(nl^2(l+1)^{2(12l+14)^2}2^{[4(l+1)^{12l+14}+2(l+1)^{24l+28}]}\big)\]
when restricted to planar graphs.
\end{theorem}

\begin{proof}
Let $G$ be a planar graph. Using Bodlaender's algorithm~\cite{bod:linalg} we can determine in time $O\big(2^{\theta(l^3)}n\big)$ if $G$ has tree-width at most $12l+13$. If so then the algorithm will also find a corresponding tree-decomposition if one exists and then the algorithm from~\cite{eggemann:thesis,eggemann-noble:ENDM} discussed at the beginning of Section~\ref{secalg} may be used to determine whether $G$ can be oriented so that its diameter is at most $l$.

On the other hand if $G$ has tree-width at least $12l+14$, then by Theorem~\ref{minor}, it has a $(2l+3)\times (2l+3)$-grid-minor and therefore by Lemma~\ref{le:diamh} its diameter is at least $l+1$ and hence the diameter of any orientation is at least $l+1$.
\end{proof}

\section{Conclusion}
We have shown that Problem~\ref{prob:chva} which is NP-complete for arbitrary graphs becomes decidable in polynomial time for graphs belonging to any minor-closed family that does not contain all apex graphs and in particular planar graphs, even if the constant two is replaced by any larger constant. It would be interesting to try to find a more efficient algorithm for this problem, not depending on graph minor theory, and also to determine the complexity when $l$ is part of the input. Furthermore it also remains to determine the complexity of minimizing the Wiener index for planar graphs.

\section{Acknowledgement}
We would like to thank Martin Lowson for agreeing to be a host to the first author on her placement, for useful discussions and for describing to us the problem motivating our work.
\bibliographystyle{plain}

\end{document}